\documentclass[reqno]{amsart}

\usepackage[margin=1in]{geometry}

\usepackage{amssymb}
\usepackage{amsmath}
\usepackage{amsthm}

\usepackage{graphicx}

\usepackage{microtype}

 \usepackage{enumerate}
\usepackage{tikz}
\usepackage{subcaption}

\usepackage[font=normalsize,labelfont=bf]{caption}

\usetikzlibrary{matrix,arrows.meta}
\usetikzlibrary{positioning,decorations.pathreplacing,decorations.markings}

\usepackage{tikz-cd}

\DeclareMathOperator{\ReP}{Re}
\DeclareMathOperator{\diag}{diag}

\newcommand{\R}{\mathbb{R}}
\newcommand{\Z}{\mathbb{Z}}

\newtheorem{definition}{Definition}[section]
\newtheorem{theorem}[definition]{Theorem}
\newtheorem{lemma}[definition]{Lemma}
\newtheorem{proposition}[definition]{Proposition}
\newtheorem{remark}[definition]{Remark}
\newtheorem{example}[definition]{Example}
\newtheorem{corollary}[definition]{Corollary}

\begin{document}

\title[Linear recurrences for non-LC independence polynomials of trees]{Linear recurrences for non-log-concave independence polynomials of trees}

\author[C. Bautista-Ramos
  \and
        C. Guill\'en-Galv\'an \and P. G\'omez-Salgado]{C\'esar Bautista-Ramos
  \and
        Carlos Guill\'en-Galv\'an \and Paulino G\'omez-Salgado 
}

\address{C. Bautista-Ramos, Facultad de Ciencias de la Computaci\'on, Benem\'erita Universidad Aut\'onoma de Puebla, 1CCO3-303, 14 sur y Av. San Claudio, Ciudad Universitaria 72570, Puebla, Pue. Mexico}
\email{cesar.bautista@correo.buap.mx}           
\address{          
 C. Guill\'en-Galv\'an  \and P. G\'omez-Salgado, Facultad de Ciencias F\'{\i}sico-Matem\'aticas, Benem\'erita Universidad Aut\'onoma de Puebla, FM7-207, 18 Sur y Av. San Claudio, Ciudad Universitaria 72570, Puebla, Pue. Mexico
}
 \email{cguillen@fcfm.buap.mx, paulino.gomezs@correo.buap.mx}

\begin{abstract}
We identify a structural pattern in the construction of known infinite families of trees whose independence polynomials are not log-concave. Using this pattern and properties of polynomial ring ideals, we derive linear recurrences for these polynomials. As a consequence, we prove that the set of non-isolated limit points of their zeros lies on the circle $|z+1/3|=1/3$ in the complex plane. Building on these recurrences, we also exhibit infinite families of trees whose independence polynomials break log-concavity at one, two, and three consecutive indices, as well as finite families that break log-concavity at four and five consecutive indices. Our approach suggests that arbitrarily many consecutive breaks may be achievable, offering further insight into a question posed by Galvin [D. Galvin, \emph{Trees with non log-concave independent set sequences}, arXiv:2502.10654v1, 2025].
\end{abstract}

\maketitle



\section{Introduction} The conjecture that independence polynomials of trees are unimodal \cite{Alavi} was strengthened to log-concavity \cite{Levit}, but subsequently disproven by the discovery of infinite families of trees failing this property at a single index \cite{Kadrawi,Kadrawi2,Galvin}. Addressing the question posed by Galvin \cite{Galvin} regarding multiple violations, recent work \cite{yo} exhibited trees with multiple \emph{non-consecutive} breaks. In this paper, we present the first families of trees exhibiting \emph{consecutive} violations. Specifically, we construct infinite families breaking log-concavity at two and three consecutive indices, and finite families breaking it at four and five. Our methods suggest that arbitrarily many consecutive violations are possible.

We achieve this by identifying a structural pattern common to prior counterexamples (Definition \ref{def:patt}). By deriving linear recurrences for these pattern graphs (Theorems \ref{thm:0} and \ref{thm:recM}), we systematically generate trees with the desired properties, distinguishing our work from sporadic counterexamples \cite{Ramos}.

The paper is organized as follows. Section \ref{sec:Br1} formalizes the unifying pattern and establishes preliminary recurrences. Section \ref{s:recPatterns} derives sharper recurrences using properties of polynomial ideals. Finally, Section \ref{s:app} details the applications, including alternative proofs of the non-log-concavity of the trees in \cite{Kadrawi,Kadrawi2}, the location of non-isolated limit points of the zeros, and the construction of trees with multiple consecutive log-concavity violations.

\section{Patterns and their linear recurrences}\label{sec:Br1}

We adopt standard definitions and notation concerning independent sets of simple graphs and the log-concavity of sequences (see, for instance, \cite{Galvin,Kadrawi}). For a graph $G$ and a vertex $v$ of $G$, let $N[v]$ denote the closed neighborhood of $v$, and let $G\setminus v$ and $G\setminus N[v]$ denote the graphs obtained by deleting $v$ and $N[v]$, respectively. We write $G^v$ for the graph $G$ rooted at vertex $v$. When the root is clear or irrelevant, we may drop the superscript and simply write $G$. Finally, we denote the independence polynomial of $G$ in the indeterminate $x$ by $I(G;x)$. Its degree equals the independence number of $G$, denoted by $\alpha(G)$.

\subsection{Pattern graphs}

Kadrawi et al.~\cite{Kadrawi} identified two infinite families of trees (with structures $3,k,k$ and $3^*,k,k+1$ for $k\geq 4$) whose independence polynomials are not log-concave. These families are shown in Figures~\ref{fa:3kk} and~\ref{fb:3kk+1}, respectively. Additional infinite families with non-log-concave independence polynomials appear in~\cite{yo,Galvin,Kadrawi2}. A common structural pattern underlying the construction of these graphs is captured by the following definition.

\begin{figure}[th]

\centerline{
\subfloat[Structure 3,$k$,$k$] {
  \resizebox{5cm}{!}{
    \begin{tikzpicture}[underbrace style/.style={
            decorate,
            decoration={
                brace,
                amplitude=5pt, 
                raise=5pt,    
                mirror        
            },
            thick,
            color=black
        }]
    \node[circle, draw] (a1) at (5,0) {};
     \node[circle, draw] (a2) at (5,1) {};
  \node[circle, draw] (b1) at (6,0) {};
  \node[circle, draw] (b2) at (6,1) {};
 \node[circle, draw] (c1) at (7,0) {};
 \node[circle, draw] (c2) at (7,1) {};
 \node[circle, draw] (i1) at (6,2) {};

  \node[circle, draw] (A1) at (8,0) {};
     \node[circle, draw] (A2) at (8,1) {};
  \node[circle, draw] (B1) at (9,0) {};
  \node[circle, draw] (B2) at (9,1) {};
 \node[circle, draw] (C1) at (11,0) {};
 \node[circle, draw] (C2) at (11,1) {};
 \node[circle, draw] (I1) at (9,2) {};

  \node[circle, draw] (D1) at (12,0) {};
     \node[circle, draw] (D2) at (12,1) {};
  \node[circle, draw] (E1) at (13,0) {};
  \node[circle, draw] (E2) at (13,1) {};
 \node[circle, draw] (F1) at (15,0) {};
 \node[circle, draw] (F2) at (15,1) {};
 \node[circle, draw] (G1) at (13,2) {};

 \node[circle, draw] (P) at (9,4) {};

 \node (pts2) at (14,1) {$\cdots$};
 \node (pts1) at (10,1) {$\cdots$};
 \node (txt1) at (9,-1) {\hspace{.9cm}{\LARGE $k$}};
 \node (txt2) at (13,-1) {\hspace{.9cm}{\LARGE $k$}};

\draw (i1)  -- (a2);
\draw (i1)  -- (b2);
\draw (i1)  -- (c2); 
  \draw (a1)  -- (a2);
  \draw (b1) -- (b2);
  \draw (c1)  -- (c2);

  \draw (I1)  -- (A2);
 \draw (I1)  -- (B2);
\draw (I1)  -- (C2); 
  \draw (A1)  -- (A2);
  \draw (B1) -- (B2);
  \draw (C1)  -- (C2);

  \draw (G1)  -- (D2);
 \draw (G1)  -- (E2);
\draw (G1)  -- (F2); 
  \draw (D1)  -- (D2);
  \draw (E1) -- (E2);
  \draw (F1)  -- (F2);

  \draw (P)  -- (i1);
  \draw (P)  -- (I1);
  \draw (P)  -- (G1);

  \draw[underbrace style] (A1.south west) -- (C1.south east);
  \draw[underbrace style] (D1.south west) -- (F1.south east);
  
\end{tikzpicture}
}\label{fa:3kk}
}
\qquad
\subfloat[Structure $3*,k,k+1$] {
\resizebox{5cm}{!}{
    \begin{tikzpicture}[fullcircle/.style={circle, draw=#1, fill=#1},
    fullcircle/.default=black,underbrace style/.style={
            decorate,
            decoration={
                brace,
                amplitude=5pt, 
                raise=5pt,    
                mirror        
            },
            thick,
            color=black
        }]
    \node[circle, draw] (a1) at (5,0) {};
     \node[circle, draw] (a2) at (5,1) {};
  \node[circle, draw] (b1) at (6,0) {};
  \node[circle, draw] (b2) at (6,1) {};
 \node[circle, draw] (c1) at (7,0) {};
 \node[circle, draw] (c2) at (7,1) {};
 \node[circle, draw] (i1) at (6,2) {};
 \node[circle, draw] (j1) at (5,-1) {};
 \node[circle, draw] (k1) at (5,-2) {};

  \node[circle, draw] (A1) at (8,0) {};
     \node[circle, draw] (A2) at (8,1) {};
  \node[circle, draw] (B1) at (9,0) {};
  \node[circle, draw] (B2) at (9,1) {};
 \node[circle, draw] (C1) at (11,0) {};
 \node[circle, draw] (C2) at (11,1) {};
 \node[circle, draw] (I1) at (9,2) {};

  \node[circle, draw] (D1) at (12,0) {};
     \node[circle, draw] (D2) at (12,1) {};
  \node[circle, draw] (E1) at (13,0) {};
  \node[circle, draw] (E2) at (13,1) {};
 \node[circle, draw] (F1) at (15,0) {};
 \node[circle, draw] (F2) at (15,1) {};
 \node[circle, draw] (G1) at (13,2) {};

 \node[circle, draw] (P) at (9,4) {};

 \node (pts2) at (14,1) {$\cdots$};
 \node (pts1) at (10,1) {$\cdots$};
 \node (txt1) at (9,-1) {\hspace{.9cm}{\LARGE $k$}};
 \node (txt2) at (13,-1) {\hspace{.9cm}{\LARGE $k+1$}};

\draw (i1)  -- (a2);
\draw (i1)  -- (b2);
\draw (i1)  -- (c2); 
  \draw (a1)  -- (a2);
  \draw (b1) -- (b2);
  \draw (c1)  -- (c2);
  \draw (j1)  -- (k1);
  \draw (a1)  -- (j1);
  
  \draw (I1)  -- (A2);
 \draw (I1)  -- (B2);
\draw (I1)  -- (C2); 
  \draw (A1)  -- (A2);
  \draw (B1) -- (B2);
  \draw (C1)  -- (C2);

  \draw (G1)  -- (D2);
 \draw (G1)  -- (E2);
\draw (G1)  -- (F2); 
  \draw (D1)  -- (D2);
  \draw (E1) -- (E2);
  \draw (F1)  -- (F2);

  \draw (P)  -- (i1);
  \draw (P)  -- (I1);
  \draw (P)  -- (G1);

  \draw[underbrace style] (A1.south west) -- (C1.south east);
  \draw[underbrace style] (D1.south west) -- (F1.south east);
  
\end{tikzpicture}
}\label{fb:3kk+1}
}

}

\caption{Trees with non-log-concave independence polynomials, for $k\geq 4$.}
\label{fig:T1}       
\end{figure}
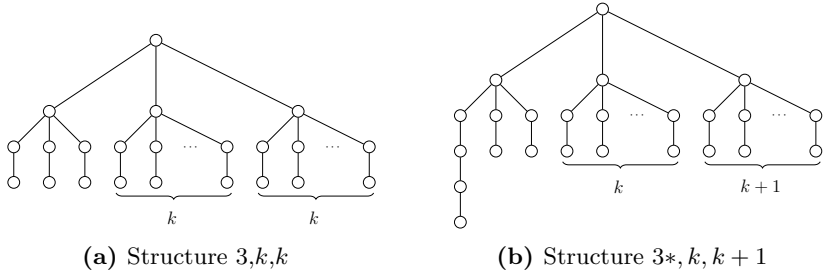

\begin{definition}\label{def:patt}
  Let $G^v$ and $H^w$ be two disjoint graphs rooted at $v$ and $w$, respectively. Let $k$ be a nonnegative integer.
  \begin{enumerate}[(i)]
    \item \label{def:patti} The \emph{scale graph of $H^w$}, denoted by $Z_k(H^w)$, is the graph constructed by taking a new vertex $v_0$ and $k$ disjoint copies of $H^w$, and joining $v_0$ to the root $w$ of each copy, as shown in Figure~\ref{fa:scale}.

    \item The graph $(G^v:H^w)_k$ is constructed by taking the disjoint union of the graph $G^v$ with the scale graph $Z_k(H^w)$ and joining the root $v$ of $G^v$ to the root $v_0$ of $Z_k(H^w)$. We view the resulting graph as rooted at $v$. See Figure~\ref{fb:patt}.

    \item If $i_1,\ldots,i_j$ is a sequence of nonnegative integers, we define a new graph recursively by
    \begin{equation}
      \label{eq:PattRec}
      (G^v:H^w)_{i_1,\ldots,i_j}=((G^v:H^w)_{i_1,\ldots,i_{j-1}}^v:H^w)_{i_j}.
    \end{equation}
    Furthermore, for any pair of nonnegative integers $m,k$, let
    \[
    (G^v:H^w)_k^{(m)}=(G^v:H^w)_{\underbrace{k,\ldots,k}_{m\text{ times}}}.
    \]
    We call these constructions \emph{pattern graphs with base graph $G$ and pendant graph $H$}.
  \end{enumerate}
\end{definition}

The following rooted trees are central to constructing infinite families of trees whose independence polynomials are not log-concave. For integers $m,t \geq 0$, we define:
\begin{itemize}
    \item $P_t^w$: path on $t$ vertices rooted at a leaf $w$;
    \item $S_{2,t}^w$: starlike graph rooted at center $w$, defined by $S_{2,t}=(P_1^w:P_1^w)_1^{(t)}$;
    \item $T_{m,t}^{v}$: tree defined by $T_{m,t}=(P_1^v:P_{2}^w)_t^{(m)}$.
\end{itemize}

Notably, the trees in~\cite{Kadrawi} correspond to $(T_{1,3}^v:P_2^w)_{k,k}$ and $(T_K^v:P_2^w)_{k,k+1}$. Here, $T_K^v$ is a tree rooted at $v$ with a single child, to which a pendant 4-vertex path and two pendant 2-vertex paths are attached. Table~\ref{tab:patt} summarizes how trees from~\cite{yo,Kadrawi2} also fit this framework. The trees $T_{m,t}$ were originally introduced by Galvin~\cite{Galvin} to resolve a conjecture in~\cite{Kadrawi2}.

  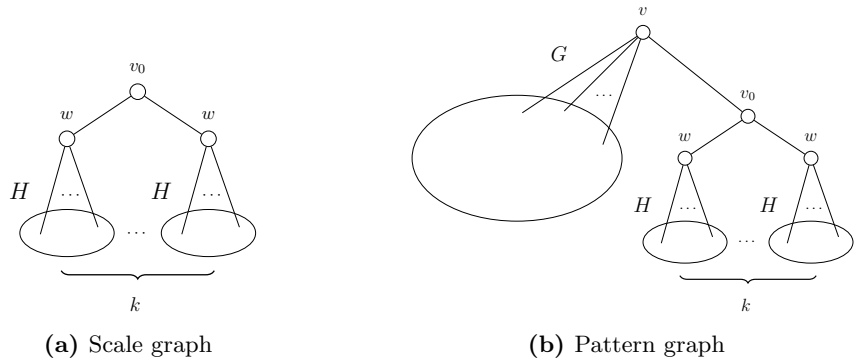
\begin{figure}[hbt]

\centerline{

\subfloat[Scale graph]{\resizebox{3.5cm}{!}{
    \begin{tikzpicture}[fullcircle/.style={circle, draw=#1, fill=#1},
    fullcircle/.default=black,underbrace style/.style={
            decorate,
            decoration={
                brace,
                amplitude=5pt, 
                raise=5pt,    
                mirror        
            },
            thick,
            color=black
        }]
       
    \node[circle,draw] (a1) at (5.5,1) {};
\node (v0) at (5.5,1.5){{\large $v_0$}};

    \node[circle, draw](b1) at  (4,0){};
    \node (w1) at (4,0.5){{\large $w$}};
    
   \node (c1) at (3.4,-2.15) {};
    \draw (4,-2) ellipse (1cm and .5cm);
    \node (H1) at (3,-1.1){{\LARGE $H$}};
    \node (pts3) at (4.1,-1.2) {$\cdots$};
    \node (c2) at (4.7,-2) {};

    \node[circle, draw](b2) at  (7,0){};
    \node (w2) at (7,0.5){{\large $w$}};
    
   \node (c3) at (6.4,-2.15) {};
    \draw (7,-2) ellipse (1cm and .5cm);
    \node (H2) at (6,-1.1){{\LARGE $H$}};
    \node (pts4) at (7.1,-1.2) {$\cdots$};
    \node (c4) at (7.7,-2) {};

     \node[circle] (d1) at (4,-2.5) {};

     \node[circle] (d3) at (7,-2.5) {};

    \node (pts1) at (5.5,-2) {$\cdots$};
  \node (txt1) at (5,-3.5) {\hspace{.9cm}{\Large $k$}};

   \draw (a1) -- (b1);
   \draw (a1) -- (b2);
 \draw (b1)  -- (c1);
 \draw (b1)  -- (c2);

  \draw (b2) -- (c3);
  \draw (b2) -- (c4);

   \draw[underbrace style] (d1.south west) -- (d3.south east);

\end{tikzpicture}
}\label{fa:scale}}

\qquad\qquad\quad

\subfloat[Pattern graph] {
  \resizebox{6cm}{!}{
    \begin{tikzpicture}[fullcircle/.style={circle, draw=#1, fill=#1},
    fullcircle/.default=black,underbrace style/.style={
            decorate,
            decoration={
                brace,
                amplitude=5pt, 
                raise=5pt,    
                mirror        
            },
            thick,
            color=black
        }]

    \node[circle, draw] (p1) at (3,3) {};
   \node (v) at (3,3.5){{\Large $v$}};
   
    \node (G0) at (0,1) {};
    \node (G1) at (1,1) {};
    \node (G2) at (2,.2) {};

    \node (G) at (1,2.5){{\LARGE $G$}};
    
    \node (pts2) at (2,1.5) {\hspace{.2cm}$\cdots$};
    
    \node[circle,draw] (a1) at (5.5,1) {};
\node (u) at (5.5,1.5){{\Large $v_0$}};

    \node[circle, draw](b1) at  (4,0){};
    \node (w1) at (4,0.5){{\Large $w$}};
    
   \node (c1) at (3.4,-2.15) {};
    \draw (4,-2) ellipse (1cm and .5cm);
    \node (H1) at (3,-1.1){{\LARGE $H$}};
    \node (pts3) at (4.1,-1.2) {$\cdots$};
    \node (c2) at (4.7,-2) {};

    \node[circle, draw](b2) at  (7,0){};
    \node (w2) at (7,0.5){{\Large $w$}};
    
   \node (c3) at (6.4,-2.15) {};
    \draw (7,-2) ellipse (1cm and .5cm);
    \node (H2) at (6,-1.1){{\LARGE $H$}};
    \node (pts4) at (7.1,-1.2) {$\cdots$};
    \node (c4) at (7.7,-2) {};

     \node[circle] (d1) at (4,-2.5) {};
   
     \node[circle] (d3) at (7,-2.5) {};

    \node (pts1) at (5.5,-2) {$\cdots$};
  \node (txt1) at (5,-3.5) {\hspace{.9cm}{\Large $k$}};


  \draw (p1)  -- (G0);

  \draw (p1)  -- (G1);
   \draw (p1)  -- (G2);

   \draw (p1)  -- (a1);
   \draw (a1) -- (b1);
   \draw (a1) -- (b2);
 \draw (b1)  -- (c1);
 \draw (b1)  -- (c2);
 
  \draw (b2) -- (c3);
  \draw (b2) -- (c4);

   \draw[underbrace style] (d1.south west) -- (d3.south east);

 \draw (0,0) ellipse (2.5cm and 1.5cm);

\end{tikzpicture}
}
\label{fb:patt}}

}
\caption{The scale graph $Z_k(H^w)$ (left) and the pattern graph $(G^v:H^w)_k$ (right), with base graph $G$ rooted at $v$ and pendant graph $H$ rooted at $w$.}
\label{fig:patt} 
\end{figure}

\subsection{Linear recurrences for the pattern graphs}

We recall standard recurrences for independence polynomials~\cite{Arocha,Gutman}. For a simple graph $G$ and a vertex $v$ of $G$,
\begin{equation}
\label{eq:Aro}
I(G; x) = I(G \setminus v; x) + x I(G \setminus N[v]; x).
\end{equation}
If $G$ is the disjoint union of graphs $G_1$ and $G_2$, then
\[
I(G;x)=I(G_1;x)I(G_2;x).
\]
To simplify notation, we identify a graph $G$ with its independence polynomial $I(G;x)$. Thus, we write $G$ in place of $I(G;x)$ when there is no risk of confusion. For instance, we might write $\alpha(G)=\deg G$. Furthermore, when applying formula~\eqref{eq:Aro} at a vertex $v$, we indicate that vertex in figures by a filled circle.

For example, the polynomial $I((G^v:H^w)_k;x)$ can be computed using the independence polynomials of the graphs shown in Figure~\ref{fig:Ps}. That is,
\begin{equation}
\label{eq:nRec}
(G^v:H^w)_k = (H\setminus w)(G^v:H^w)_{k-1} + x (H\setminus N_H[w]) \, H^{k-1} G.
\end{equation}
Consequently, we have
\begin{align*}
(G^v:H^w)_k - (H\setminus w)(G^v&:H^w)_{k-1} = H \left( x(H\setminus N_H[w]) H^{k-2} G \right) \\
&= H (G^v:H^w)_{k-1} - H(H\setminus w) (G^v:H^w)_{k-2}.
\end{align*}
This yields the following result.

\begin{theorem}\label{thm:0}
Let $G^v$ and $H^w$ be simple rooted graphs. For all $k \geq 2$,
\[
(G^v:H^w)_k = \bigl( (H\setminus w) + H \bigr)(G^v:H^w)_{k-1} - (H\setminus w)H \,(G^v:H^w)_{k-2}.
\]
This is a linear recurrence with characteristic polynomial
\[
\chi(r) = \bigl(r - (H\setminus w)\bigr)(r - H).
\]
\end{theorem}

\begin{table}[t]
  
  \caption{Patterns of known infinite families with non-log-concave independence polynomials}
\label{tab:patt}     

\begin{tabular}[tb]{lll}
\hline\noalign{\smallskip}
Structure & Pattern   \\
  \noalign{\smallskip}\hline\noalign{\smallskip}
  $3,k,k$ & $(T_{1,3}^v:P_2^w)_{k}^{(2)}$  \\
$3,k,k+1$ & $(T_{1,3}^v:P_2^w)_{k,k+1}$  \\
  $3,k,k+2$ & $(T_{1,3}^v:P_2^w)_{k,k+2}$  \\
  $3^*,k,k$ & $(T_K^v:P_2^w)_{k}^{(2)}$  \\
  $3^*,k,k+1$ & $(T_K^v:P_2^w)_{k,k+1}$  \\
  $3^*,k,k+2$ & $(T_K^v:P_2^w)_{k,k+2}$  \\
  $3^*,k,k+3$ & $(T_K^v:P_2^w)_{k,k+3}$  \\
  $T_{m,t}$ & $(P_1^v:P_2^w)_{t}^{(m)}$  \\
  $GT_{m,t}$ & $(P_2^v:S_{2,t}^w)_3^{(m)}$\\
\noalign{\smallskip}\hline
\end{tabular}
\end{table}

While Theorem~\ref{thm:0} concerns the sequence generated by increasing $k$, we obtain a similar recurrence when increasing the recursion depth $m$.

\begin{theorem}\label{thm:recM}
    Let $G^v$ and $H^w$ be simple rooted graphs. For any fixed nonnegative integer $k$, the sequence of independence polynomials $\bigl((G^v:H^w)_k^{(m)}\bigr)_{m\geq 0}$ satisfies a linear recurrence whose characteristic polynomial is
    \[
    \chi(r) = \bigl(r - H^k\bigr)\bigl(r - Z_k(H^w)\bigr).
    \]
\end{theorem}

\begin{proof}
    The proof is analogous to that of Theorem~\ref{thm:0}. It proceeds by applying recurrence~\eqref{eq:Aro} to the vertex indicated in Figure~\ref{fig:m}.
\end{proof}

\section{Polynomial rings, recursive patterns, and their linear recurrences}\label{s:recPatterns}

In this section, we use the framework of polynomial rings to compute the independence polynomials of pattern graphs defined in Definition~\ref{def:patt}. We begin by recalling the module structure of linearly recurrent sequences, summarized in the following well-known lemma.

\begin{lemma}\label{l:mod}
    Let $R$ be a commutative ring with identity and let $M$ be an $R$-module. Let $P$ be the direct product $\prod_{i=0}^\infty M_i$, where $M_i=M$ for all $i$. For a polynomial $f(r) \in R[r]$, let $\mathcal{S}(f)$ be the set of sequences in $P$ that satisfy the linear recurrence with characteristic polynomial $f(r)$. Then $\mathcal{S}(f)$ is an $R$-submodule of $P$.
\end{lemma}

Building on the module structure from Lemma~\ref{l:mod}, we now construct a specific quotient ring that models the recurrence in Theorem~\ref{thm:0}.

\begin{lemma}\label{l:zAlg}
    Let $a$ and $b$ be formal variables. Define $R_0$ as the quotient of the polynomial ring in the variables $z_0, z_1, \ldots$ with coefficients in $\Z[a,b]$ by the ideal generated by the set $\{z_k - a z_{k-1} + b z_{k-2} \mid k \geq 2\}$. That is,
    \[
        R_0 = \Z[a,b][z_0, z_1, \ldots] \big/ \langle z_k - a z_{k-1} + b z_{k-2} \mid k \geq 2 \rangle.
    \]
    Then the following hold:
    \begin{enumerate}[(i)]
        \item \label{l:zAlgi} The sequence $(z_k^2)_{k \geq 0}$ in $R_0$ is $\Z[a,b]$-linearly recurrent with characteristic polynomial
              $\widehat{\chi}_0(r) = (r - b)\bigl(r^2 + (2b - a^2)r + b^2\bigr).
              $
        \item \label{l:zAlgii} For any fixed integer $m \geq 1$, the sequence $(z_k z_{k+m})_{k \geq 0}$ in $R_0$ is $\Z[a,b]$-linearly recurrent with characteristic polynomial
              $
                  \widehat{\chi}_1(r) = (r + b) \widehat{\chi}_{0}(r).
              $
    \end{enumerate}
\end{lemma}

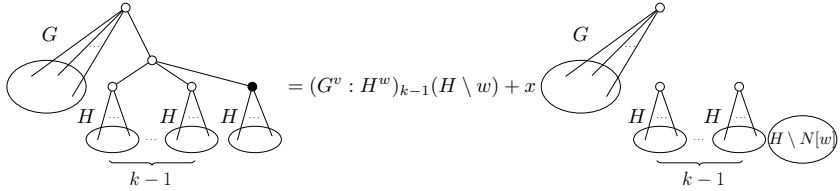
\begin{figure}[bh]
  \resizebox{12cm}{!}{
    \begin{tikzpicture}[fullcircle/.style={circle, draw=#1, fill=#1},
    fullcircle/.default=black,underbrace style/.style={
            decorate,
            decoration={
                brace,
                amplitude=5pt, 
                raise=5pt,    
                mirror        
            },
            thick,
            color=black
        }]
    \node[circle, draw] (p1) at (3,3) {};
   
    \node (G0) at (-0.7,0.3) {};
    \node (G1) at (0.3,0.3) {};
    \node (G2) at (0.9,-0.5) {};

    \node (G) at (0,2){\scalebox{2.2}{ $G$}};
    
    \node (pts2) at (1.85,1.5) {$\cdots$};


    \node[circle,draw] (a1) at (4,1) {};

    \node[circle, draw](b1) at  (2.5,0){};
    
   \node (c1) at (1.9,-2.15) {};
    \draw (2.5,-2) ellipse (1cm and .5cm);
    \node (H1) at (1.5,-1.1){\scalebox{2.2}{$H$}};
    \node (pts3) at (2.6,-1.2) {$\cdots$};
    \node (c2) at (3.2,-2) {};

    \node[circle, draw](b2) at  (5.5,0){};
    
   \node (c3) at (4.9,-2.15) {};
    \draw (5.5,-2) ellipse (1cm and .5cm);
    \node (H2) at (4.5,-1.1){\scalebox{2.2}{ $H$}};
    \node (pts4) at (5.6,-1.2) {$\cdots$};
    \node (c4) at (6.2,-2) {};

  \node[circle, fill, draw](B2) at  (7.8,0){};
    
   \node (C3) at (7.2,-2.15) {};
    \draw (7.9,-2) ellipse (1cm and .5cm);
    \node (H3) at (6.8,-1.1){\scalebox{2.2}{ $H$}};
    \node (pts5) at (7.9,-1.2) {$\cdots$};
    \node (C4) at (8.6,-2) {};

     \node[circle] (d1) at (2.5,-2.5) {};
     \node[circle] (d3) at (5.5,-2.5) {};

    \node (pts1) at (4,-2) {$\cdots$};
  \node (txt1) at (3.5,-3.5) {\hspace{.9cm}{\Huge $k-1$}};



  \draw (p1)  -- (G0);

  \draw (p1)  -- (G1);
   \draw (p1)  -- (G2);

   \draw (p1)  -- (a1);
   \draw (a1) -- (b1);
   \draw (a1) -- (b2);
 \draw (b1)  -- (c1);
 \draw (b1)  -- (c2);

  \draw (b2) -- (c3);
  \draw (b2) -- (c4);
  
 \draw (a1)-- (B2);
  \draw (B2) -- (C3);
   \draw (B2) -- (C4); 
 
   \draw[underbrace style] (d1.south west) -- (d3.south east);

 \draw (0,0) ellipse (1.5cm and 1cm);
\end{tikzpicture}

\begin{tikzpicture}[fullcircle/.style={circle, draw=#1, fill=#1},
    fullcircle/.default=black,underbrace style/.style={
            decorate,
            decoration={
                brace,
                amplitude=5pt, 
                raise=5pt,    
                mirror        
            },
            thick,
            color=black
        }]
  \node[circle, draw] (p1) at (0,3) {};
   
    \node (G0) at (-3.7,.3) {};
    \node (G1) at (-2.7,.3) {};
    \node (G2) at (-2.1,-.5) {};

    \node (G) at (-3,2){\scalebox{2.2}{$G$}};
    
    \node (pts2) at (-1.1,1.5) {$\cdots$};

    
   \node[circle] (a1) at (2.5,1) {};

   \node (eq) at (-9.4,0) {\scalebox{2.2}{$=(G^v:H^w)_{k-1}(H\setminus w)+x$}}; 
   
   \node[circle, draw](b1) at  (0,0){};
    
   \node (c1) at (-0.6,-2.15) {};
    \draw (0,-2) ellipse (1cm and .5cm);
    \node (H1) at (-1,-1.1){\scalebox{2.2}{ $H$}};
    \node (pts3) at (0.1,-1.2) {$\cdots$};
    \node (c2) at (0.7,-2) {};

    \node[circle, draw](b2) at  (3,0){};
    
   \node (c3) at (2.4,-2.15) {};
    \draw (3,-2) ellipse (1cm and .5cm);
    \node (H2) at (2,-1.1){\scalebox{2.2}{$H$}};
    \node (pts4) at (3.1,-1.2) {$\cdots$};
    \node (c4) at (3.7,-2) {};

  \node[circle](B2) at  (3.3,0){};
    
   \node (C3) at (4.7,-2.15) {};
    \draw (5.4,-2) ellipse (1.3cm and .9cm);
    \node (H3) at (5.3,-2){\scalebox{1.8}{ $H\setminus N[w]$}};
    \node (C4) at (9.1,-2) {};

     \node[circle] (d1) at (0,-2.5) {};
     \node[circle] (d3) at (3,-2.5) {};

    \node (pts1) at (1.5,-2) {$\cdots$};
  \node (txt1) at (1.2,-3.5) {\hspace{.9cm}{\Huge $k-1$}};


  \draw (p1)  -- (G0);

  \draw (p1)  -- (G1);
   \draw (p1)  -- (G2);

 \draw (b1)  -- (c1);
 \draw (b1)  -- (c2);

  \draw (b2) -- (c3);
  \draw (b2) -- (c4);
  
 
   \draw[underbrace style] (d1.south west) -- (d3.south east);

 \draw (-3,0) ellipse (1.5 cm and 1cm);

\end{tikzpicture}
}
\caption{Calculation of the independence polynomial of $(G^v:H^w)_k$ by applying Eq.~\eqref{eq:Aro} at the root $w$ (filled circle) of a copy of the pendant graph $H$.}
\label{fig:Ps}       

\end{figure}

\begin{proof}
    The following relations hold in $R_0$:
    \begin{enumerate}[(i)]
        \item Let $k \geq 3$. Starting from the defining recurrence, we have:
        \begin{align}
            z_k^2 &= (a z_{k-1} - b z_{k-2})^2 \nonumber \\
                  &= a^2 z_{k-1}^2 - 2ab\,z_{k-1}z_{k-2} + b^2 z_{k-2}^2 \label{eq:z2} \\
                  &= a^2 z_{k-1}^2 - 2ab\,z_{k-2}(a z_{k-2} - b z_{k-3}) + b^2 z_{k-2}^2 \nonumber \\
                  &= a^2 z_{k-1}^2 + b(b - 2a^2) z_{k-2}^2 + 2ab^2 z_{k-3}z_{k-2}. \label{eq:fin}
        \end{align}
        To eliminate the mixed term $z_{k-3}z_{k-2}$, we apply~\eqref{eq:z2} shifted to index $k-1$, which gives:
        \[
            2ab\,z_{k-2}z_{k-3} = -z_{k-1}^2 + a^2 z_{k-2}^2 + b^2 z_{k-3}^2.
        \]
        Multiplying this by $b$ and substituting into~\eqref{eq:fin} yields:
        \begin{equation}
            \label{eq:rec}
            z_k^2 = (a^2 - b) z_{k-1}^2 + (b^2 - a^2 b) z_{k-2}^2 + b^3 z_{k-3}^2,
        \end{equation}
        which is a linear recurrence with characteristic polynomial $\widehat{\chi}_0(r)$.
 
\item For $m=1$, the relation $z_{k+1} = a z_k - b z_{k-1}$ implies
\[
    z_{k}z_{k+1} + b z_{k-1}z_k = a z_k^2.
\]
Consequently, the sum $z_{k}z_{k+1} + b z_{k-1}z_k$ satisfies the same recurrence as the sequence $(z_{k}^2)_{k\geq 0}$ in \eqref{eq:rec}; that is,
\begin{multline*}
    z_{k}z_{k+1} + b z_{k-1}z_k
    = (a^2-b)\left(z_{k-1}z_k + b z_{k-2}z_{k-1}\right) \\
    + (b^2-a^2b)\left(z_{k-2}z_{k-1} + b z_{k-3}z_{k-2}\right) + b^3\left(z_{k-3}z_{k-2} + b z_{k-4}z_{k-3}\right).
\end{multline*}
Rearranging terms to isolate $z_{k}z_{k+1}$ yields
\begin{equation}
    \label{eq:zkRec}
    z_{k}z_{k+1} = \left(a^2-2b\right) z_{k-1}z_k + b^2(2b-a^2) z_{k-3}z_{k-2} + b^4 z_{k-4}z_{k-3},
\end{equation}
which is a linear recurrence with characteristic polynomial $\widehat{\chi}_1(r)$.

For $m=2$, first observe that $\widehat{\chi}_1(r) = (r+b)\widehat{\chi}_0(r)$, so the sequence $(z_k^2)_{k \geq 0}$ is annihilated by $\widehat{\chi}_1(r)$. From the identity
\[
    z_k z_{k+2} = a z_k z_{k+1} - b z_k^2,
\]
we see that $(z_k z_{k+2})_{k \geq 0}$ is a $\Z[a,b]$-linear combination of $(z_k z_{k+1})_{k \geq 0}$ and $(z_k^2)_{k \geq 0}$. Since both sequences belong to the module $\mathcal{S}(\widehat{\chi}_1)$ (by the case $m=1$ and the observation above), Lemma~\ref{l:mod} implies that $(z_k z_{k+2})_{k \geq 0}$ is also in $\mathcal{S}(\widehat{\chi}_1)$. For $m > 2$, the result follows by induction using the relation $z_k z_{k+m} = a z_k z_{k+m-1} - b z_k z_{k+m-2}$ and Lemma~\ref{l:mod}. 
 
        \end{enumerate}
\end{proof}

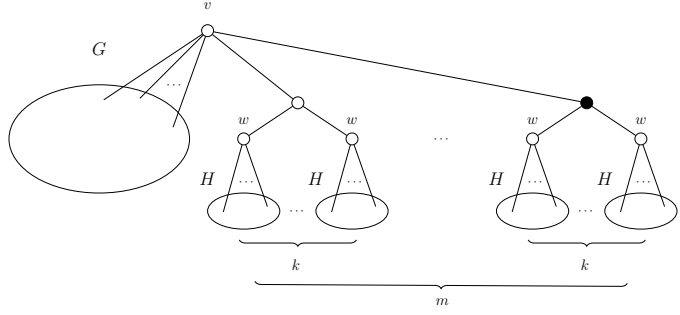
\begin{figure}[t]
  \resizebox{9cm}{!}{
    \begin{tikzpicture}[fullcircle/.style={circle, draw=#1, fill=#1},
    fullcircle/.default=black,underbrace style/.style={
            decorate,
            decoration={
                brace,
                amplitude=5pt, 
                raise=5pt,    
                mirror        
            },
            thick,
            color=black
        }]

    \node[circle, draw] (p1) at (3,3) {};
    
   \node (v) at (3,3.7){{\Large $v$}};
   
    \node (G0) at (0,1) {};
    \node (G1) at (1,1) {};
    \node (G2) at (2,.2) {};

    \node (G) at (0,2.5){{\LARGE $G$}};
    
    \node (pts2) at (2,1.5) {\hspace{.2cm}$\cdots$};
    
    \node[circle,draw] (a1) at (5.5,1) {};

    \node[circle, draw](b1) at  (4,0){};
    \node (w1) at (4,0.5){{\Large $w$}};

   \node (c1) at (3.4,-2.15) {};
    \draw (4,-2) ellipse (1cm and .5cm);
    \node (H1) at (3,-1.1){{\LARGE $H$}};
    \node (pts3) at (4.1,-1.2) {$\cdots$};
    \node (c2) at (4.7,-2) {};

    \node[circle, draw](b2) at  (7,0){};
    \node (w2) at (7,0.5){{\Large $w$}};

   \node (c3) at (6.4,-2.15) {};
    \draw (7,-2) ellipse (1cm and .5cm);
    \node (H2) at (6,-1.1){{\LARGE $H$}};
    \node (pts4) at (7.1,-1.2) {$\cdots$};
    \node (c4) at (7.7,-2) {};

     \node[circle] (d1) at (4,-2.5) {};
     \node[circle] (d3) at (7,-2.5) {};

    \node (pts1) at (5.5,-2) {$\cdots$};
  \node (txt1) at (5,-3.5) {\hspace{.9cm}{\Large $k$}};


  \draw (p1)  -- (G0);

  \draw (p1)  -- (G1);
   \draw (p1)  -- (G2);

   \draw (p1)  -- (a1);
   \draw (a1) -- (b1);
   \draw (a1) -- (b2);
 \draw (b1)  -- (c1);
 \draw (b1)  -- (c2);
  \draw (b2) -- (c3);
  \draw (b2) -- (c4);
  

   \draw[underbrace style] (d1.south west) -- (d3.south east);

 \draw (0,0) ellipse (2.5cm and 1.5cm);
    \node[circle,fill,draw] (A1) at (13.5,1) {};

    \node[circle, draw](B1) at  (12,0){};
    \node (w3) at (12,0.5){{\Large $w$}};

   \node (C1) at (11.4,-2.15) {};
    \draw (12,-2) ellipse (1cm and .5cm);
    \node (HH1) at (11,-1.1){{\LARGE $H$}};
    \node (Pts3) at (12.1,-1.2) {$\cdots$};
    \node (C2) at (12.7,-2) {};

    \node[circle,draw](B2) at  (15,0){};
    \node (w4) at (15,0.5){{\Large $w$}};
    
   \node (C3) at (14.4,-2.15) {};
    \draw (15,-2) ellipse (1cm and .5cm);
    \node (HH2) at (14,-1.1){{\LARGE $H$}};
    \node (Pts4) at (15.1,-1.2) {$\cdots$};
    \node (C4) at (15.7,-2) {};

     \node[circle] (D1) at (12,-2.5) {};
   
     \node[circle] (D3) at (15,-2.5) {};

    \node (Pts1) at (13.5,-2) {$\cdots$};
  \node (Txt1) at (13,-3.5) {\hspace{.9cm}{\Large $k$}};
  \draw[underbrace style] (D1.south west) -- (D3.south east);
  
\draw (p1) -- (A1);
\draw (A1) -- (B1);
\draw (A1) -- (B2);
\draw (B1) -- (C1);
\draw (B1) -- (C2);
\draw (B2) -- (C3);
\draw (B2) -- (C4);

\node (Pts6) at (9.5,0) {$\cdots$};

\node (E2) at (14.5,-3.7) {};

\draw[underbrace style] (txt1.south west) -- (E2.south east);
 \node (Txt3) at (9.5,-4.5) {{\Large $m$}};
\end{tikzpicture}
}
\caption{The pattern graph $(G^v:H^w)_k^{(m)}$. Eq.~\eqref{eq:Aro} is applied to the filled vertex $v$.}
    \label{fig:m}
  \end{figure}

General recurrence theory~\cite{Cerruti} yields weaker results in this setting. Applying~\cite[Cor.~2.7]{Cerruti} using the Kronecker product of the companion matrix for $(z_k)_{k\geq 0}$ produces the characteristic polynomial $(r-b)^2(r^2 + (2b - a^2)r + b^2)$ for the squared sequence $(z_k^2)_{k\geq 0}$. This is a nontrivial multiple of the sharper polynomial $\widehat{\chi}_0(r)$ derived in Lemma~\ref{l:zAlg}\eqref{l:zAlgi}.

The following lemma provides a framework for recursively computing the independence polynomials of pattern graphs of the type shown in Figure~\ref{fig:m}.

\begin{lemma}\label{l:recm}
    Let $R_1 = \Z[x][z_0, z_1]$ be the polynomial ring in the variables $z_0, z_1$ with coefficients in $\Z[x]$. Let $p_1$ and $p_2$ be distinct polynomials in $\Z[x]$. Define a sequence $(z_k)_{k \geq 0}$ in $R_1$ recursively by
    \[
        z_k = (p_1 + p_2) z_{k-1} - p_1 p_2 z_{k-2} \quad \text{for } k \geq 2,
    \]
    with initial terms given by the variables $z_0$ and $z_1$. Then, for each positive integer $m$, the sequence $(z_k^m)_{k \geq 0}$ is linearly recurrent with characteristic polynomial
    \[
        \chi_m(r) = \prod_{j=0}^m \bigl(r - p_1^{\,m-j} p_2^{\,j}\bigr).
    \]
\end{lemma}

\begin{proof}
    By definition, the sequence $(z_k)_{k\geq 0}$ is linearly recurrent with characteristic polynomial $\chi_1(r) = (r - p_1)(r - p_2)$, whose roots $p_1$ and $p_2$ are distinct. Consequently, for all $k \geq 0$,
    \[
        z_k = c_1 p_1^k + c_2 p_2^k,
    \]
    where $c_1$ and $c_2$ belong to the field of fractions $F$ of $R_1$. Applying the binomial theorem gives
    \begin{equation}
        \label{eq:zm}
        z_k^m = \sum_{j=0}^m \binom{m}{j} c_1^{\,m-j} c_2^{\,j} \bigl(p_1^{\,m-j} p_2^{\,j}\bigr)^k.
    \end{equation}
    Each term $p_1^{\,m-j} p_2^{\,j}$ is a root of $\chi_m(r)$. Therefore, for each $0 \leq j \leq m$, the sequence $\bigl((p_1^{m-j} p_2^{j})^k\bigr)_{k \geq 0}$ is linearly recurrent and belongs to the $F$-vector space $\mathcal{S}(\chi_m)$. Since $(z_k^m)_{k \geq 0}$ is expressed in~\eqref{eq:zm} as an $F$-linear combination of these sequences, Lemma~\ref{l:mod} implies that $(z_k^m)_{k \geq 0}$ also lies in $\mathcal{S}(\chi_m)$.
\end{proof}

The following lemma establishes an isomorphism between the ring $R_1$ from Lemma~\ref{l:recm} and a specialization of the quotient ring introduced in Lemma~\ref{l:zAlg}.

\begin{lemma}\label{l:iso}
    Let $R_1$ and $p_1, p_2$ be as in Lemma~\ref{l:recm}, and let $J$ be the ideal of $\Z[x][z_0, z_1, \ldots]$ generated by
    \[
        \{z_k - (p_1 + p_2)z_{k-1} + p_1 p_2 z_{k-2} \mid k \geq 2\}.
    \]
    Then the $\Z[x]$-algebra homomorphism $\eta: R_1 \to \Z[x][z_0, z_1, \ldots]/J$, defined by $\eta(z_0) = [z_0]$ and $\eta(z_1) = [z_1]$, is an isomorphism.
\end{lemma}

\begin{proof}
    The quotient algebra $\Z[x][z_0, z_1, \ldots]/J$, together with the classes $[z_0]$ and $[z_1]$, satisfies the universal property of the polynomial ring $R_1$ together with its variables $z_0,z_1$.
\end{proof}

The main application of the preceding algebraic framework is the following result.

\begin{theorem}\label{thm:Lrec}
    Let $G^v$ and $H^w$ be two simple rooted graphs, and let $m$ be a nonnegative integer. Then the sequences of independence polynomials
    \[
        \left((G^v:H^w)_{k}^{(m)}\right)_{k\geq 0} \quad \text{and} \quad \left((G^v:H^w)_{k,\,k+m+1}\right)_{k\geq 0}
    \]
    are linearly recurrent. Their characteristic polynomials are
    \[
        \chi_{m}(r) = \prod_{j=0}^{m} \bigl(r - (H\setminus w)^{\,m-j} H^{\,j}\bigr)
        \quad \text{and} \quad 
        \chi(r) = \bigl(r + (H\setminus w) H\bigr) \chi_2(r),
    \]
    respectively.
\end{theorem}

\begin{proof}
    Let $p_1 = I(H \setminus w; x)$ and $p_2 = I(H; x)$. Let $R_0$ and $R_1$ be the $\Z[x]$-algebras from Lemmas~\ref{l:zAlg} and~\ref{l:recm}, respectively, and let $\mathcal{M}$ denote the set of all monomials in the variables $z_0, z_1, \ldots$. Observe that for any sequence of nonnegative integers $i_0, \ldots, i_j$ and any permutation $\sigma$ of $\{0, \ldots, j\}$, there is a graph isomorphism
    \[
        (G^v:H^w)_{i_0,\ldots,i_j} \cong (G^v:H^w)_{i_{\sigma(0)},\ldots,i_{\sigma(j)}}.
    \]
    Consequently, the map $f : \mathcal{M} \to \Z[x]$ defined by
    \[
        z_{i_0} \cdots z_{i_j} \mapsto I\bigl((G^v:H^w)_{i_0,\ldots,i_j}; x\bigr)
    \]
    is well defined. This map extends uniquely to a $\Z[x]$-algebra homomorphism $\phi: \Z[x][z_0, z_1, \ldots] \to \Z[x]$.

    Now consider the following commutative diagram, where $\pi_0$ and $\pi_1$ are the canonical projection homomorphisms:
    \[
    \begin{tikzcd}[column sep=large, row sep=large]
        & \mathcal{M} \arrow[hook]{d} \arrow["f"]{dr} & \\
        \Z[a,b][z_0,z_1,\ldots] \arrow["\pi_0"]{d} \arrow["\varphi"]{r} & \Z[x][z_0,z_1,\ldots] \arrow["\phi",swap]{r} \arrow["\pi_1"]{d} & \Z[x] \\
        R_0 \arrow["\widetilde{\varphi}",swap]{r}       & \Z[x][z_0,z_1,\ldots]/J \arrow[swap,"\widetilde{\phi}"]{ur} & R_1 \arrow[tail, two heads, "\eta"']{l}
    \end{tikzcd}
    \]
    Here, $\varphi$ is the unique $\Z$-algebra homomorphism determined by $a \mapsto p_1+p_2$, $b \mapsto p_1 p_2$, and $z_k \mapsto z_k$ for all $k$. The induced homomorphism $\widetilde{\phi}$ exists because the recursive relation~\eqref{eq:PattRec} and Theorem~\ref{thm:0} imply that for any indices with $i_k \geq 2$,
    \[
        \phi(z_{i_0}\ldots z_{i_{k}}) = (p_1+p_2)\phi(z_{i_0}\ldots z_{i_{k}-1}) - p_1 p_2 \phi(z_{i_0}\ldots z_{i_{k}-2}),
    \]
    which shows that $J \subseteq \ker \phi$. By the commutativity of the diagram, we have
    \[
        (\widetilde{\phi} \circ \widetilde{\varphi})\bigl(\pi_0(z_k z_{k+m+1})\bigr) = I\bigl((G^v:H^w)_{k,k+m+1}; x\bigr)
    \]
    and, together with Lemma \ref{l:iso}
    \[
        (\widetilde{\phi} \circ \eta)(z_k^m) = I\bigl((G^v:H^w)_{k}^{(m)}; x\bigr).
    \]
    The result now follows from Lemma~\ref{l:recm}.
\end{proof}

We apply finite-dimensional representations of the formal power series algebra $\R[[x]]$ to solve the recurrences in Theorems~\ref{thm:recM} and \ref{thm:Lrec} for specific coefficients. Throughout, $M_n(R)$ denotes the algebra of $n \times n$ matrices over a commutative ring $R$.

\begin{definition}
    Let $u$ be a nonnegative integer. Let $\psi_u \colon \R[[x]] \to M_{u+1}(\R)$ be the $\R$-algebra homomorphism defined by
    \[
        \sum_{j=0}^\infty a_j x^j \mapsto
        \begin{pmatrix}
            a_0 & 0 & \cdots & 0 \\
            a_1 & a_0 & \cdots & 0 \\
            \vdots & \vdots & \ddots & \vdots \\
            a_u & a_{u-1} & \cdots & a_0
        \end{pmatrix}.
    \]
    For an additional positive integer $k$, denote by $\Psi_{k,u}$ the induced $\R$-algebra homomorphism
\[
    \Psi_{k,u} \colon M_{k}\bigl(\R[x]\bigr) \to M_{k(u+1)}(\R), \qquad \bigl(q_{i,j}(x)\bigr)_{i,j} \mapsto \bigl(\psi_u(q_{i,j})\bigr)_{i,j}.
\]
\end{definition}

\begin{proposition}\label{p:PattG}
    Let $(p_m)_{m\geq 0}$ be a sequence in $\R[[x]]$ that is linearly recurrent with characteristic polynomial $\chi(r)$ of degree $k$. Let $u$ be a nonnegative integer. Write $p_m(x)=\sum_{j=0}^\infty a_{m,j}x^j$ and set $\mathbf{v}_m=(a_{m,0},\ldots,a_{m,u})^\top$. Then, for all $m \geq k$,
    \[
        \begin{pmatrix}
            \mathbf{v}_{m-k+1}\\
            \vdots\\
            \mathbf{v}_{m-1}\\
            \mathbf{v}_{m}
        \end{pmatrix}
        = \mathbf{Q} \mathbf{J}^{\,m-k+1} \mathbf{Q}^{-1}
        \begin{pmatrix}
            \mathbf{v}_0\\
            \vdots\\
            \mathbf{v}_{k-1}
        \end{pmatrix},
    \]
    where $\mathbf{J}$ is the Jordan normal form of the image under $\Psi_{k,u}$ of the companion matrix of $\chi(r)$, and $\mathbf{Q}$ is an invertible matrix whose columns form a Jordan basis of $\R^{k(u+1)}$.
\end{proposition}

\begin{proof}
    Let $\chi(r) = r^k - \sum_{j=0}^{k-1} c_j r^j$, where each $c_j \in \R[x]$. Let $\mathbf{w} = (1, 0, \ldots, 0)^\top$ a column vector in $\R^{u+1}$. Using the $\R[[x]]$-module structure of $\R^{u+1}$ induced by the representation $\psi_u$, we obtain, for each integer $m \geq k$,
    \[
        p_m \mathbf{w} = \sum_{j=1}^{k} c_{k-j} \, p_{m-j} \mathbf{w}.
    \]
    Since $\psi_u(p_j)\mathbf{w} = \mathbf{v}_j$ for all $j$, this becomes
    \begin{equation}\label{eq:dJorA}
        \mathbf{v}_m = \sum_{j=1}^{k} c_{k-j} \, \mathbf{v}_{m-j}.
    \end{equation}
    Equation \eqref{eq:dJorA} is equivalent to the block matrix identity
    \[
        \begin{pmatrix}
            \mathbf{v}_{m-k+1} \\
            \mathbf{v}_{m-k+2} \\
            \vdots \\
            \mathbf{v}_{m}
        \end{pmatrix}
        = \Psi_{k,u}(\mathbf{M})
        \begin{pmatrix}
            \mathbf{v}_{m-k} \\
            \mathbf{v}_{m-k+1} \\
            \vdots \\
            \mathbf{v}_{m-1}
        \end{pmatrix},
    \]
    where $\mathbf{M}$ is the companion matrix of $\chi(r)$. The result follows from the Jordan normal form of $\Psi_{k,u}(\mathbf{M})$.
\end{proof}

\begin{remark}\label{r:degiii} 
Let $k,m$ be positive integers, and let $J_k(\lambda)$ be the $k\times k$ Jordan block with eigenvalue $\lambda$. Then the $(i,j)$-entry of $J_k(\lambda)^m$ is
$
\binom{m}{j-i}\lambda^{\,m+i-j}$,
 $1\leq i,j\leq k.
$

\end{remark}

\section{Applications}\label{s:app}
\subsection{Non-isolated limit points of zeros}

All families in Table~\ref{tab:patt}, with the exception of the general family $GT_{m,t}$, are constructed using the pendant graph $P_2^w$. The specific case $GT_{1,t}$, however, fits this pattern, as it is isomorphic to $(P_3^v:P_2^w)_t^{(3)}$. We prove that for these families, the sets of non-isolated limit points of the zeros of their independence polynomials coincide and form a circle. Recall that these families are known to break log-concavity at a single index~\cite{Kadrawi,Kadrawi2,Galvin,yo}.

\begin{theorem}\label{thm:nonI}
    Let $(G_k)_{k \geq 0}$ be the family $GT_{1,k}$ or any family in Table~\ref{tab:patt} with base graph $T_{1,3}^v$ or $T_K^v$. Then the set of non-isolated limit points of the zeros of their independence polynomials is the circle $|x + 1/3| = 1/3$ in the complex plane. The same property holds for the family $(T_{m,t})_{t \geq 0}$, provided there exists an integer $t_0$ such that the independence polynomial of $T_{m,t_0}$ is not log-concave.
\end{theorem}

\begin{proof}
    Let $p_1=I(P_1;x)$ and $p_2=I(P_2;x)$. By Theorem~\ref{thm:Lrec} with $H=P_2$, the sequence $(I(G_k;x))_{k\geq 0}$ satisfies a linear recurrence with characteristic roots in the set $\{p_1^j p_2^{m-j}, -p_1 p_2\}$. If the minimal polynomial has a single root $\lambda$, then $I(G_k;x) = I(G_0;x)\lambda^k$. However, $\lambda = -p_1 p_2$ implies negative coefficients, which is impossible, while $\lambda = p_1^j p_2^{m-j}$ implies that $I(G_k;x)$ is log-concave (being a product of log-concave polynomials), contradicting established results~\cite{Kadrawi2,Kadrawi,yo}. Thus, the minimal polynomial must have at least two distinct roots. By the Beraha–Kahane–Weiss theorem~\cite{Beraha}, the non-isolated limit points of the zeros satisfy the equimodular condition $|\lambda_i(x)| = |\lambda_j(x)|$ for distinct roots. Dividing by common factors reduces this to $|p_1| = |p_2|$, or explicitly $|1+x|=|1+2x|$. By applying the inversion map $x\mapsto w=1/x$, this equation becomes $|1+w|=|2+w|$, which describes the line $\ReP w=-3/2$. This line maps under inversion onto the circle $|x+1/3|=1/3$. 
\end{proof}

The existence of integers $m$ and $t_0$ for which the independence polynomial of $T_{m,t_0}$ is non-log-concave is guaranteed by~\cite[Thm.~2.1]{Galvin}.

Despite Theorem~\ref{thm:nonI}, the property that the non-isolated limit points of the zeros form the circle $|x+1/3|=1/3$ is not a necessary condition for an infinite family of trees to break log-concavity in their independence polynomials at one index. We provide a counterexample in Corollary~\ref{cor:cNonI}.

\subsection{Trees breaking log-concavity at one index}

The linear recurrences derived in Theorem~\ref{thm:Lrec} provide an alternative proof of the result from \cite{Kadrawi} that trees with the structure $3,k,k$ have non-log-concave independence polynomials. To analyze the failure of log-concavity near the leading terms of these polynomials, it is convenient to work with their \emph{reflected} forms.

\begin{definition}
    Let $p \in \R[x]$ be a polynomial of degree $n$. The \emph{reflected polynomial} of $p$, denoted by $Rp$, is defined as
    $
        (Rp)(x) = x^n p(1/x).
    $
\end{definition}

From the definition, it follows that the reflection operator $R$ is multiplicative: for any polynomials $p, q \in \R[x]$,
\[
    R(pq) = (Rp)(Rq).
\]
Additionally, $R$ satisfies an additive property under degree constraints: if $\deg p \ge \deg q$ and $\deg(p+q) = \deg p$, then
\[
    R(p+q) = Rp + x^{\deg p - \deg q} Rq.
\]

\begin{definition}
    Let $(p_k)_{k\geq 0}$ be a linearly recurrent sequence of polynomials in $\R[x]$ with characteristic polynomial $\chi(r)=r^n-a_{n-1} r^{n-1}-\cdots-a_0$. We call such a sequence \emph{reflection-compatible} if the reflected sequence $(Rp_k)_{k\geq 0}$ is linearly recurrent with characteristic polynomial $\chi_R(r)=r^n-Ra_{n-1}r^{n-1}-\cdots-Ra_0$.
\end{definition}

\begin{lemma}\label{l:Rcomp}
    Let $(p_k)_{k\geq 0}$ be a linearly recurrent sequence of polynomials in $\R[x]$ with characteristic polynomial $\chi(r)=r^n-a_{1} r^{n-1}-\cdots-a_n$. If there exist integers $\ell$ and $r$ such that for each nonnegative integer $k$, $\deg p_k=\ell k+r$, and for each $j=1,\ldots,n$, $\deg a_j=\ell j$, then the sequence $(p_k)_{k\geq 0}$ is reflection-compatible.
\end{lemma}

\begin{proof}
    If $k\geq n$, then $p_k=\sum_{j=1}^n a_{j}p_{k-j}$. From the hypothesis, it follows that each term on the right-hand side of this equation has the same degree: $\deg p_k$. Then, the additive and multiplicative properties of the reflection operator $R$ lead to the result.
\end{proof}

\begin{example}\label{e:1}
    Let $G^v$ and $H^w$ be two rooted simple graphs. For an arbitrary fixed positive integer $k$, define $U_m = (G^v:H^w)_k^{(m)}$.
    If $\alpha(H) > \alpha(H\setminus w)$ and $\alpha(G\setminus v) > \alpha(G\setminus N[v])$, then the sequence of independence polynomials $(U_m)_{m\geq 0}$ is reflection-compatible.
\end{example}

\begin{proof}
    Applying \eqref{eq:Aro} to the scale graph $Z_k(H^w)$ at the vertex $v_0$ (see Definition~\ref{def:patt}\eqref{def:patti}) gives
    \[
        Z_k(H^w) = H^k + x(H\setminus w)^k.
    \]
    Since $\alpha(H) \ge 1 + \alpha(H\setminus w)$, we have $k\alpha(H) \ge 1 + k\alpha(H\setminus w)$; therefore $\alpha(Z_k(H^w)) = k\alpha(H)$. Applying \eqref{eq:Aro} again, now to the graph $(G^v:H^w)_k^{(m)}$ at the vertex $v$, we obtain
    \begin{equation}\label{eq:PattU}
        U_m = (G\setminus v)\,Z_k(H^w)^m + x\,(G\setminus N[v])\,H^{km}.
    \end{equation}
    Consequently, $\deg U_m = km\alpha(H) + \alpha(G\setminus v)$.
    Theorem~\ref{thm:recM} tells us that the sequence $(U_m)_{m\geq 0}$ is linearly recurrent with characteristic polynomial
    \[
        \chi(r) = r^2 - e_1\!\bigl(H^k,Z_k(H^w)\bigr)r +e_2\!\bigl(H^k,Z_k(H^w)\bigr),
    \]
    where $e_1$ and $e_2$ are the elementary symmetric polynomials in two variables.
    Since $\deg e_j\!\bigl(H^k,Z_k(H^w)\bigr) = jk\alpha(H)$ for $j=1,2$, we may apply Lemma~\ref{l:Rcomp} to conclude that the sequence is reflection-compatible.
\end{proof}

\begin{example}\label{e:2}
    Let $G^v$ and $H^w$ be two rooted simple graphs and let $m$ be an arbitrary fixed nonnegative integer. Let $L_k=(G^v:H^w)_k^{(m)}$. If $\alpha(G\setminus v)>\alpha(G\setminus N[v])$ and $\alpha(H)=\alpha(H\setminus w)$, then the sequence of independence polynomials $(L_k)_{k\geq 0}$ is reflection-compatible.
\end{example}

\begin{proof}
    Proceeding as in the proof of Example \ref{e:1}, from the equation $Z_k(H^w)=H^k+x(H\setminus w)^k$ and the condition $\alpha(H)=\alpha(H\setminus w)$, we obtain $\alpha(Z_k(H^w))=1+k\alpha(H)$. This implies that $\alpha(L_k)=mk\alpha(H)+m+\alpha(G\setminus v)$, by \eqref{eq:PattU}.
    Further, from Theorem \ref{thm:Lrec}, we know that the sequence $(L_k)_{k\geq 0}$ is linearly recurrent with characteristic polynomial
    \[
        \chi(r)=r^{m+1}-\widetilde{e}_1r^m+\cdots+(-1)^{m+1}\widetilde{e}_{m+1},
    \]
    where $\widetilde{e}_j=e_j\bigl((H\setminus w)^m, (H\setminus w)^{m-1}H,\ldots,H^m\bigr)$ for $j=1,\ldots,m+1$, and $e_1,e_2,\ldots$ are the elementary symmetric polynomials. Each $\deg \widetilde{e}_j=jm\alpha(H)$. Thus, since $m$ is fixed, we can use Lemma \ref{l:Rcomp} to obtain the result.
\end{proof}

To apply these reflection properties, we need to compute the degrees of the relevant independence polynomials.

\begin{lemma}\label{l:deg}
    Let $k$ and $m$ be nonnegative integers.
    \begin{enumerate}[(i)]
        \item\label{l:degi} The independence polynomial of $(T_{1,3}^v:P_2^w)_{k,k+m}$ is monic and has degree $2k + m + 6$.
        \item\label{l:degii} The independence polynomial of $S_{2,k}$ is monic and has degree $k + 1$.
    \end{enumerate}
\end{lemma}

\begin{proof}
    The unique maximum independent set in the tree $(T_{1,3}^v:P_2^w)_{k,k+m}$ consists of the three children of the root together with their $2k+m+3$ grandchildren (see Figure~\ref{fig:T1} for the case $m=0$), giving a total of $2k+m+6$ vertices. Therefore, the independence polynomial is monic of this degree. For $S_{2,k} = (P_1^v:P_1^w)_1^{(k)}$, the unique maximum independent set consists of the root $v$ of its base graph together with its $k$ grandchildren. Hence, its independence polynomial is monic of degree $k + 1$.
\end{proof}

\begin{corollary}[Kadrawi et al.~\cite{Kadrawi}]\label{cor:Ka}
  For any $k \geq 4$, the independence polynomial of the tree with structure $3,k,k$ has degree $2k+6$ and is not log-concave; log-concavity is broken at index $2k+5$.
\end{corollary}

\begin{proof}
    Let $L_k = (T_{1,3}^v:P_2^w)_{k}^{(2)}$. Since
    \[
        \alpha(T_{1,3}^v\setminus v) = \alpha(S_{2,3}) = 4 > 3 = \deg(P_1^3) = \alpha(T_{1,3}^v\setminus N[v])
    \]
and $\alpha(P_2^w\setminus w) = \alpha(P_1) = \alpha(P_2)$, Example~\ref{e:2} implies that $(RL_k)_{k\geq 0}$ is linearly recurrent with characteristic polynomial $\prod_{j=0}^2 (r - RP_1^{\,2-j} RP_2^{\,j})$. Write $RL_k = \sum_{j=0}^\infty a_{k,j} x^j$. Applying Proposition~\ref{p:PattG} with $u=2$ to this reflected sequence gives
    \begin{equation}\label{eq:3kk}
        \begin{pmatrix}
            \mathbf{v}_{k-2} \\
            \mathbf{v}_{k-1} \\
            \mathbf{v}_k
        \end{pmatrix}
        = \mathbf{Q} \diag\bigl( J_3(1)^{k-2},\; J_3(2)^{k-2},\; J_3(4)^{k-2} \bigr) \mathbf{Q}^{-1}
        \begin{pmatrix}
            \mathbf{v}_0 \\
            \mathbf{v}_1 \\
            \mathbf{v}_2
        \end{pmatrix},
    \end{equation}
    where
    \[
        \mathbf{v}_k = \begin{pmatrix}
            a_{k,0} \\ a_{k,1} \\ a_{k,2}
        \end{pmatrix}, \qquad
        \mathbf{Q} = \begin{pmatrix}0 & 0 & 32 & 0 & 0 & 8 & 0 & 0 & 2\\0 & 64 & -128 & 0 & 24 & -24 & 0 & 8 & -4\\128 & -224 & 320 & 72 & -64 & 46 & 32 & -14 & 5\\0 & 0 & 32 & 0 & 0 & 16 & 0 & 0 & 8\\0 & 64 & -64 & 0 & 48 & -24 & 0 & 32 & -8\\128 & -96 & 96 & 144 & -56 & 28 & 128 & -24 & 6\\0 & 0 & 32 & 0 & 0 & 32 & 0 & 0 & 32\\0 & 64 & 0 & 0 & 96 & 0 & 0 & 128 & 0\\128 & 32 & 0 & 288 & 32 & 0 & 512 & 32 & 0\end{pmatrix},
    \]
    and the initial conditions can be calculated from \eqref{eq:Aro} for $L_k$ at the root $v$, yielding $\mathbf{v}_0 = (1,13,46)^\top$, $\mathbf{v}_1 = (1,17,124)^\top$, and $\mathbf{v}_2 = (1,23,321)^\top$. Next, using Remark \ref{r:degiii}  and straightforward calculations involving the right-hand side of \eqref{eq:3kk}, we obtain $a_{k,0}=1$, $a_{k,1}=2^{k+1}+2k+11$, and
\[
    a_{k,2} = 9 \cdot 2^{2k} + 11 \cdot 2^{k+1} + 3k 2^{k} + 2k^2 + 21k + 15.
\]
Consequently,
\[
    a_{k,0}a_{k,2} - a_{k,1}^2 = 5 \cdot 2^{2k} - (5k + 22) 2^{k} - 2k^2 - 23k - 106.
\]
For $k \geq 4$, the inequalities $3 \cdot 2^{2k} > (5k+22)2^k$ and $2 \cdot 2^{2k} > 2k^2 + 23k + 106$ hold (which can be proven easily by induction), implying $a_{k,0}a_{k,2} - a_{k,1}^2 > 0$. As $\alpha(L_k) = 2k+6$, the log-concavity is broken at index $2k+5$ for all $k \geq 4$.
\end{proof}

This result shows that all trees identified by Kadrawi et al.~\cite{Kadrawi} and by Kadrawi and Levit~\cite{Kadrawi2} that are constructed using the base graph $T_{1,3}^v$ and the pendant graph $P_2^w$ are not log-concave for sufficiently large $k$. In fact, the following corollary generalizes this observation. In its proof, we use the exact-order asymptotic notation $\Theta$ (see~\cite[p.~87]{Brassard} or~\cite[p.~110]{Knuth}).

\begin{corollary}\label{cor:Br1}
    Let $m$ be a nonnegative integer. Then the independence polynomial of the tree $(T_{1,3}^v:P_2^w)_{k,k+m}$ is not log-concave for all sufficiently large $k$.
\end{corollary}
\begin{proof}
    By Lemma~\ref{l:deg}\eqref{l:degi}, we can write $R(T_{1,3}^v:P_2^w)_{k,k+m} = 1 + b_{k,m}x + a_{k,m}x^2 + \cdots$, where $b_{k,m}, a_{k,m} \in \Z$. Therefore, it suffices to prove that $a_{k,m} - b_{k,m}^2 \in \Theta(2^{2k})$. We establish this, together with $b_{k,m} \in \Theta(2^{k})$, by induction on $m$. The base case $m=0$ follows from the proof of Corollary~\ref{cor:Ka}. Now assume $m>0$. From formula~\eqref{eq:nRec}, Lemma~\ref{l:deg}, and the additive and multiplicative properties of the reflection operator $R$, we obtain
\[
    R(T_{1,3}^v:P_2^w)_{k,k+m} = P_1 \, R(T_{1,3}^v:P_2^w)_{k,k+m-1} + x^{k+m+2}(x+2)^{k-1} R T_{1,3}.
\]
Thus, for all $k\geq 0$, the term involving $x^{k+m+2}$ vanishes modulo $x^3$, yielding
\[
    a_{k,m}x^2 + b_{k,m}x + 1 \equiv (a_{k,m-1} + b_{k,m-1})x^2 + (b_{k,m-1} + 1)x + 1 \pmod{x^3}.
\]
Equating coefficients and applying the induction hypothesis gives
\[
    a_{k,m} = a_{k,m-1} + b_{k,m-1} \quad \text{and} \quad b_{k,m} = b_{k,m-1} + 1 \in \Theta(2^k).
\]
Consequently,
\[
    a_{k,m} - b_{k,m}^2 = (a_{k,m-1} - b_{k,m-1}^2) - (b_{k,m-1}+1) \in \Theta(2^{2k}).
\]

\end{proof}

The inductive technique used in the proof of Corollary~\ref{cor:Br1} applies similarly to pattern graphs with base graph $T_K^v$ and pendant graph $P_2^w$.

\begin{corollary}\label{cor:Br1a}
    Let $m$ be a nonnegative integer. Then the independence polynomial of the tree $(T_K^v:P_2^w)_{k,k+m}$ is not log-concave for all sufficiently large $k$.
\end{corollary}

Additional infinite families of trees whose independence polynomials break log-concavity at a single index can be constructed without relying on the pendant graph $P_2$. First, it is convenient to summarize the key algebraic properties of the independence polynomials for pattern graphs of the form $(T_{1,\ell}^v:S_{2,n}^w)_k^{(m)}$.

\begin{theorem}\label{thm:PattG}
    Let $k, \ell, n$ be fixed nonnegative integers. For each nonnegative integer $m$, let $U_m$ denote the tree $(T_{1,\ell}^v:S_{2,n}^w)_k^{(m)}$. Then:
    \begin{enumerate}[(i)]
        \item\label{thm:PattGi} $\alpha(U_m) = k(n+1)m + \ell + 1$.
        \item\label{thm:PattGii} $RU_m = RS_{2,\ell} \,RT_{k,n}^m + RP_2^\ell \,RS_{2,n}^{km}$.
        \item\label{thm:PattGiii} The sequence $(RU_m)_{m\geq 0}$ is linearly recurrent with characteristic polynomial 
        \[
            \chi_1(r) = \bigl(r - RS_{2,n}^k\bigr)\bigl(r - RT_{k,n}\bigr).
        \]
    \end{enumerate}
\end{theorem}

\begin{proof}
    Applying formula~\eqref{eq:Aro} to the root $v$ of $U_m$ gives
    \begin{equation}\label{eq:PattG1}
        U_m = S_{2,\ell} \, T_{k,n}^m + x \, P_2^\ell \, S_{2,n}^{km}.
    \end{equation}
    From~\cite[p.~3]{Galvin} we have $\alpha(T_{k,n}) = k(n+1)$. Together with Lemma~\ref{l:deg}\eqref{l:degii}, this shows that both terms on the right-hand side of~\eqref{eq:PattG1} have the same degree, namely $k(n+1)m + \ell + 1$. This proves part~\eqref{thm:PattGi}. Because the degrees are equal, applying the reflection operator $R$ to both sides of~\eqref{eq:PattG1} is straightforward. Using $Rx = 1$ together with the multiplicativity and additivity of $R$, we obtain part~\eqref{thm:PattGii}. Finally, since $T_{1,\ell}^v$ as a base graph and $S_{2,n}$ as a pendant graph satisfy the conditions of Example~\ref{e:1}, the sequence $(U_m)_{m\geq 0}$ is reflection-compatible; that is, part~\eqref{thm:PattGiii} holds.
\end{proof}

\begin{corollary}\label{cor:cNonI}
    For $m \geq 9$, the independence polynomial of the tree $(T_{1,2}^v:S_{2,4}^w)_2^{(m)}$ has degree $10m+3$, and its log-concavity is broken at index $10m+2$.
\end{corollary}

\begin{proof}
    Let $U_m = (T_{1,2}^v:S_{2,4}^w)_2^{(m)}$. From Theorem~\ref{thm:PattG}, we obtain that $(RU_m)_{m\geq 0}$ is linearly recurrent with characteristic polynomial $(r-RS_{2,4}^2)(r-RT_{2,4})$. Applying Proposition~\ref{p:PattG} with $u=2$ yields
    \begin{equation}\label{eq:S24}
        \begin{pmatrix}
            \mathbf{v}_{m-1} \\
            \mathbf{v}_m
        \end{pmatrix}
        = \mathbf{Q} \operatorname{diag}\bigl( J_{4}(1)^{m-1},\, J_{2}(1)^{m-1} \bigr) \mathbf{Q}^{-1}
        \begin{pmatrix}
            \mathbf{v}_{0} \\
            \mathbf{v}_{1}
        \end{pmatrix},
    \end{equation}
    where $RU_m = \sum_{j=0}^\infty a_{m,j}x^j$ and
    \[
        \mathbf{v}_m = \begin{pmatrix}
            a_{m,0} \\ a_{m,1} \\ a_{m,2}
        \end{pmatrix}, \qquad
        \mathbf{Q} =\begin{pmatrix}0 & 0 & -21 & 21 & 0 & 21\\0 & -7056 & 0 & 0 & 740 & -740\\-2122176 & -290136 & 0 & 0 & 0 & 0\\0 & 0 & -21 & 0 & 0 & 21\\0 & -7056 & -7056 & 0 & 740 & 0\\-2122176 & -2412312 & -290136 & 0 & 0 & 0\end{pmatrix}.
    \]
    The initial conditions are computed from Theorem~\ref{thm:PattG}\eqref{thm:PattGii}: $\mathbf{v}_0 = (5,10,6)^\top$ and $\mathbf{v}_1 = (5,466,5346)^\top$.
    Straightforward evaluation of the right-hand side of \eqref{eq:S24} gives
    \[
        a_{m,0}=5,\quad
        a_{m,1}=456 m + 10,\quad
        a_{m,2}=47008 m^{2} - 41668 m + 6.
    \]
    Hence,
    \[
        a_{m,0}a_{m,2} - a_{m,1}^{2}
        = 238214 + 270412 (m-9) + 27104 (m-9)^{2}.
    \]
    By Theorem~\ref{thm:PattG}\eqref{thm:PattGi}, $\alpha(U_m) = 10m+3$. Since $a_{m,0}a_{m,2} > a_{m,1}^{2}$ for all $m \geq 9$, the log-concavity of $U_m$ is broken at index $10m+2$ whenever $m \geq 9$.
\end{proof}

Note that for the trees in Corollary~\ref{cor:cNonI}, the set of non-isolated limit points of the zeros of their independence polynomials is given by the equimodular equation $|S_{2,4}^2| = |T_{2,4}|$ (by Theorem~\ref{thm:recM} and the Beraha--Kahane--Weiss theorem). The locus defined by this condition does not coincide with the circle $|x+1/3|=1/3$; for example, the point $x = -2/3$ lies on the circle but does not satisfy the equimodular equation.

\subsection{Trees breaking log-concavity at two indices}\label{s:infiniteTwo}

By increasing the parameters in the pattern graphs from Corollary~\ref{cor:cNonI}, we can construct trees whose independence polynomials break log-concavity at several consecutive coefficients.

\begin{corollary}\label{cor:infiniteTwo}
    For $m \geq 16$, the independence polynomial of the tree $(T_{1,3}^v:S_{2,4}^w)_2^{(m)}$ has degree $10m+4$, and its log-concavity is broken at indices $10m+3$ and $10m+2$.
\end{corollary}

\begin{proof}
    The proof follows the same lines as that of Corollary~\ref{cor:cNonI}. 
    Let $G_m = (T_{1,3}^v:S_{2,4}^w)_2^{(m)}$, and let $a_{m,j}$ be integers such that $RG_m = \sum_{j=0}^{\alpha(G_m)} a_{m,j} x^j$. 
    From Theorem~\ref{thm:PattG} with $\ell = 3$, $n = 4$, and $k = 2$, together with Proposition~\ref{p:PattG} with $u=3$, we obtain 
    \begin{equation}\label{eq:T13}
        \begin{pmatrix}
            \mathbf{v}_{m-1} \\
            \mathbf{v}_m
        \end{pmatrix}
        = \mathbf{Q} \diag\!\bigl( J_{5}(1)^{m-1},\, J_{3}(1)^{m-1} \bigr) \mathbf{Q}^{-1}
        \begin{pmatrix}
            \mathbf{v}_{0} \\
            \mathbf{v}_{1}
        \end{pmatrix},
    \end{equation}
    where $\mathbf{Q}$ is the matrix
    \[
\left(\begin{matrix}0 & 0 & 0 & -1 & 1 & 0 & 0 & 1\\0 & 0 & -336 & 0 & 0 & 0 & \frac{62160}{1579} & - \frac{62160}{1579}\\0 & -101056 & -13816 & 0 & 0 & \frac{2190400}{1579} & 0 & 0\\\\-29976576 & -5105216 & -205840 & 0 & 0 & \frac{541587776}{1579} & 0 & 0\\0 & 0 & 0 & -1 & 0 & 0 & 0 & 1\\0 & 0 & -336 & -336 & 0 & 0 & \frac{62160}{1579} & 0\\\\0 & -101056 & -114872 & -13816 & 0 & \frac{2190400}{1579} & \frac{2190400}{1579} & 0\\\\-29976576 & -35081792 & -5311056 & -205840 & 0 & \frac{541587776}{1579} & \frac{541587776}{1579} & 0\end{matrix}\right)
    \] 
  and $\mathbf{v}_0=(9,23,21,8)^\top$, $\mathbf{v}_1=(9,639,9065,42876)^\top$ are obtained from Theorem~\ref{thm:PattG}\eqref{thm:PattGii}. Then, for all $m \geq 0$,
    \begin{gather*}
        a_{m,0} = 9, \quad a_{m,1} = 616 m + 23, \quad a_{m,2} = 50208 m^2 - 41164 m + 21, \\
        a_{m,3} = \frac{13223168}{3} m^3 - 12135360 m^2 + \frac{23311516}{3} m + 8.
    \end{gather*}
    Straightforward calculations then yield
    \begin{gather*}
        a_{m,0}a_{m,2} - a_{m,1}^2 = 4\bigl(18104 (m-6)^2 + 117545 (m-6) + 53441\bigr), \\
        \begin{aligned}
            a_{m,1}a_{m,3} - a_{m,2}^2 &= \frac{1}{3}\bigl(582941696 (m-16)^4 + 27586828800 (m-16)^3 \\
            &\quad + 437202140464 (m-16)^2 + 2355242090028 (m-16) \\
            &\quad + 552109256637 \bigr).
        \end{aligned}
    \end{gather*}
    Since $\alpha(G_m) = 10m+4$ by Theorem~\ref{thm:PattG}\eqref{thm:PattGi}, the result follows.
\end{proof}

\subsection{Log-concavity broken at three or more indices}\label{s:Br3-5}
The previous pattern-graph construction can be extended to obtain trees  breaking  log-concavity at three, four, and five consecutive indices.

\begin{corollary}\label{cor:infiniteThree}
    For $m \geq 13$, the independence polynomial of the tree $(T_{1,7}^v:S_{2,5}^w)_2^{(m)}$ has degree $12m+8$, and its log-concavity is broken at indices $12m+7$, $12m+6$, and $12m+5$.
\end{corollary}

\begin{proof}
    The argument proceeds as before. Let $G_m=(T_{1,7}^v:S_{2,5}^w)_2^{(m)}$, and let $a_{m,j}$ be the coefficients of its reflected independence polynomial: $RG_m=\sum_{j=0}^{\alpha(G_m)} a_{m,j} x^j$. Applying Theorem~\ref{thm:PattG} with $\ell = 7$, $n = 5$, and $k = 2$, together with Proposition~\ref{p:PattG} with $u=4$, yields
    \begin{gather*}
        a_{m,0} = 129, \quad a_{m,1} = 10570 m + 583, \quad a_{m,2} = 953266 m^2 - 566943 m + 1141, \\
        a_{m,3} = \frac{687810932}{3} m^3 - 583211348 m^2 + \frac{1070981623}{3} m + 1267, \\
        \begin{split}
            a_{m,4} &= \frac{182164910818}{3} m^4 - \frac{989856488450}{3} m^3 + \frac{3398341376827}{6} m^2\\
            &\quad - \frac{1782883459493}{6} m + 875.
        \end{split}
    \end{gather*}
    Consequently,
    \[
        a_{m,0}a_{m,2}-a_{m,1}^2=35895660+94482357(m-8)+11246414(m-8)^2,
    \]
    \begin{multline*}
        a_{m,1}a_{m,3}-a_{m,2}^2=17655405115342+\frac{146409650144791}{3}(m-3)\\
        +\frac{121056582494899}{3}(m-3)^2+\frac{39678207054968}{3}(m-3)^3\\
        +\frac{4544013350972}{3}(m-3)^4,
    \end{multline*}
    and
    \begin{multline*}
        a_{m,2}a_{m,4}-a_{m,3}^2=1370296158733103513864\\
        +\frac{6220680127005039907100}{3}(m-13)+\frac{12717506131617535724605}{18}(m-13)^2\\
        +\frac{214694392978422167551}{2}(m-13)^3+\frac{75665809567839162839}{9}(m-13)^4\\
        +\frac{1000049830957485238}{3}(m-13)^5+\frac{47870969448786140}{9}(m-13)^6.
    \end{multline*}
    Since all three expressions are positive for $m \geq 13$, and since $\alpha(G_m) = 12m+8$ by Theorem~\ref{thm:PattG}\eqref{thm:PattGi}, log-concavity is broken at the stated indices.
  \end{proof}
  
    Similarly, for $m \in \{164, \dots, 213\}$, the trees $(T_{1,4}^v:S_{2,10}^w)_3^{(m)}$ ($\alpha=33m+5$) break log-concavity at the four consecutive indices $33m+1, \dots, 33m-2$, which are not adjacent to the independence number. Likewise, for $m \in \{171, \dots, 214\}$, the trees $(T_{1,7}^v:S_{2,10}^w)_3^{(m)}$ ($\alpha=33m+8$) break log-concavity at the five consecutive indices $33m+3, \dots, 33m-1$.


\end{document}